\documentclass[english]{article}
\usepackage[T1]{fontenc}
\usepackage[latin9]{inputenc}
\usepackage{geometry}
\usepackage{amsthm}
\usepackage{amsmath}
\usepackage{amssymb}

\makeatletter
\theoremstyle{plain}
\newtheorem{thm}{\protect\theoremname}
  \theoremstyle{plain}
  \newtheorem{lem}[thm]{\protect\lemmaname}

\makeatother

\usepackage{babel}
  \providecommand{\lemmaname}{Lemma}
\providecommand{\theoremname}{Theorem}

\begin{document}
\global\long\def\GL{\mathrm{GL}}
\global\long\def\CC{\mathbb{C}}
\global\long\def\cF{\mathcal{F}}
\global\long\def\QQ{\mathbb{Q}}
\global\long\def\ZZ{\mathbb{Z}}
 \global\long\def\cA{\mathcal{A}}
\global\long\def\cL{\mathcal{L}}
\global\long\def\cZ{\mathcal{Z}}
\global\long\def\cO{\mathcal{O}}
\global\long\def\End{\mathrm{End}}

\title{A New Bound for the Uniform Admissibility Theorem}

\author{Alexander Kemarsky}

\maketitle
Let $G$ be a reductive group defined over a $p$-adic field $F$,
that is $G=F$-points of a reductive group defined over $F$. A representation
$(\pi,V)$ of $G$ is called smooth if for every $v\in V$ there exists
an open compact subgroup $K<G$ such that $\pi(g)v=v$ for every $g\in K$.
Let us denote by $V^{K}$ the subspace of all $K$-fixed vectors in
$V$. A representation is called admissible if $V^{K}$ is finite-dimensional
for every open compact subgroup $K<G$. In this work we consider only
smooth admissible representations of $G$. Also, algebras are associative
and with unit.\\
Let $K$ be a fixed open compact subgroup of $G$. In \cite{Bernstein},
Bernstein proved the following 
\begin{thm}[Uniform Admissibility Theorem]
\label{thm: uniform adm} There exists a uniform constant $N(G,K)$
such that for every irreducible representation $(\pi,V)$ of $G$
there is an inequality $\dim V^{K}\le N(G,K)$.
\end{thm}
In this work we reprove one of the two lemmas (Lemma \ref{lem: commutative matrices})
in Bernstein's original proof. As a result we obtain a sharper estimate
for the bound $N(G,K)$ in the theorem. For the convenience of the
reader we sketch here (almost without proofs) the main steps in Bernstein's
proof. We will demonstrate the new bound on $\dim V^{K}$ only in
the case $G=\GL_{n}(F)$. The case of a general reductive group is
similar. \\
Let us reformulate the theorem in terms of Hecke algebras. Denote
by $H(G,K)$ the Hecke algebra consisting of all compactly supported
functions $f:G\to\CC$ such that $f(k_{1}gk_{2})=f(g)$ for all $g\in G$
and $k\in K$. Denote by $H(G)$ the Hecke algebra consisting of all
locally constant and compactly supported functions $f:G\to\CC.$ Let
$(\pi,V)$ be an irreducible representation of $G$ such that $V^{K}\ne0$.
The Hecke algebra $H(G)$ acts on $(\pi,V).$ Note that the space
$V^{K}$ is an irreducible representation of $H(G,K)$. Indeed, let
$0\ne v\in V^{K}$, $0\ne w\in V^{K}.$ The representation $(\pi,V)$
is irreducible, hence there exists a function $f\in H(G)$ such that
$\pi(f)v=w$. Obviously $\pi(1_{K}*f*1_{K})v=w$ and the convolution
$1_{K}*f*1_{K}\in H(G,K)$. Thus we can reformulate Theorem \ref{thm: uniform adm}
as the following 
\begin{thm}
There exists a uniform constant $N(G,K)$ such that for every irreducible
finite-dimensional $H(G,K)$-module $W$ there is an inequality $\dim W\le N(G,K)$.
\end{thm}
The theorem follows from the following lemmas.
\begin{lem}
\cite[Proposition 2]{Bernstein}\label{lem: commutative matrices}
Let $A_{1},A_{2},...,A_{l}$ be a commuting family of matrices in
$M_{n\times n}(\CC)$. Let $\cF$ be the algebra of matrices in $M_{n\times n}(\CC)$
generated by $A_{1},A_{2},...,A_{l}$ and the identity. Then
\[
\dim\cF\le(l+1)n^{2-\frac{2}{l+1}}.
\]

\end{lem}
We will prove this lemma later.

\begin{lem}
\cite[Proposition 1]{Bernstein}\label{lem: assoc algebra} Let $\cL$
be an algebra over $\CC$, $\cA,\,\cZ$ commutative subalgebras in
$\cL$, $A_{1},\, A_{2},\,...,\, A_{l}\in\cA$, $X_{1},...,\, X_{p},\, Y_{1},\,...,\, Y_{q}\in\cL$.
Let us assume that $\cZ$ lies in the center of the algebra $\cL$,
$\cZ\subset\cA$,~ $\cA$ is the commutative algebra generated by
$A_{1},\,...,A_{l}$ and $\cZ$, and that any element $X\in\cL$ can
be written in the form $X=\sum X_{i}P_{ij}Y_{j}$, where $P_{ij}\in\cA$
,$(i=1,\,...,p;\,\, j=1,\,...,\, q).$ Then any irreducible finite
dimensional representation of the algebra $\cL$ has dimension at
most $(pq)^{(l+1)/2}(l+1)^{(l+1)/2}.$ \end{lem}
\begin{proof}
If $\rho:\,\cL\to\End(V)$ is an irreducible representation, $\dim V=n$,
then $\rho(\cZ)=\CC\cdot1$ (by Schur's Lemma), $\dim\rho(\cL)=n^{2}$(by
Burnside's theorem). By the conditions of Lemma \ref{lem: assoc algebra},
\[
\dim\rho(\cL)\le pq\,\dim\rho(\cA)\le pq\,(l+1)n^{2-2/(l+1)}.
\]
Hence $n^{2}\le pq(l+1)n^{2-2/(l+1)}$, that is $n\le(pq)^{(l+1)/2}(l+1)^{(l+1)/2}.$
\end{proof}

\begin{lem}
The Hecke algebra $H(G,K)$ satisfies the conditions of Lemma \ref{lem: assoc algebra}. 
\end{lem}
For the proof see \cite[Lemma]{Bernstein}. We only note that $l$
is the $F$-semi-simple rank of the group $G$ and the numbers $p,\, q$
depend linearly on $[K_{0}:K]$ where $K_{0}$ is a maximal compact
subgroup of $G$ such that $K\subset K_{0}.$ Let us demonstrate the
choices of $p,\, q,\, x_{i},$ and $y_{j}$ in the case $G=\GL_{n}(F)$.
Let $\cO=\left\{ x\in F\,\mid|x|\le1\right\} $ and let $\varpi$
be a generator of the maximal ideal in $\cO$. Then $K_{0}=GL_{n}(\cO)$
is a maximal compact subgroup of $G$. Let $K$ be a ``good enough''
compact open subgroup of $\GL_{n}(O_{F})$, for example a congruence
subgroup of $\GL_{n}(O_{F})$,
\[
K=K_{m}:=\left\{ x\in G\,\mid||1-x||\le||\varpi||^{m}\right\} ,
\]
where $m\ge1$ and $||x||=\max|x_{ij}|$. In this case one can take
$a_{j}$ $(j=1,\,...,\, n-1)$ to be a diagonal matrix, $(a_{j})_{ii}=1$
for $i\le j$ and $(a_{j})_{ii}=\varpi$ for $i>j$ and define $A_{j}=1_{Ka_{j}K}$.
Let $\cZ$ be the algebra generated by $1_{Kg}$ for $g\in Z(G)$.
Let $\cA$ be the algebra generated by $\cZ$ and $A_{j}$, $j=1,\,...,n-1$.
Let $p=q=[K_{0}:K]$, decompose $K_{0}=\cup Kg_{i}$ $i=1,\,...,p$
and choose $x_{i}=y_{i}=1_{Kg_{i}K}$ . The elements $x_{i},y_{j},$
$i,j=1,\,...,\, p$ and the algebras $\cA,\,\cZ,$ and $H(G,K)$ satisfy
the conditions of Lemma \ref{lem: assoc algebra}. As a consequence,
we obtain 
\[
\dim V^{K}\le[\GL_{n}(\cO_{F}):K]^{n}\cdot n^{n/2}
\]
for every irreducible representation $(\pi,V)$ of $GL_{n}(F)$. This
improves Bernstein's bound 
\[
\dim V^{K}\le[\GL_{n}(\cO_{F}):K]^{2^{n-1}}.
\]
See \cite{Bernstein} for more details. In the proof of Lemma \ref{lem: commutative matrices}
we need the following
\begin{lem}
\label{lem: jordan} Let $A\in M_{n\times n}(\CC)$ be a nilpotent
matrix and let $m\ge1$. Then 
\[
\dim\left(Span\left\{ A^{m}B\,\mid AB=BA\right\} \right)\le\frac{n^{2}}{m}.
\]
\end{lem}
\begin{proof}
Suppose $A=diag(J_{l_{1}},J_{l_{2}},...,J_{l_{k}})$ where $J_{i}$
is a Jordan block of dimension $i\times i$. The assertion $AC=CA$
for $C_{n\times n}=(C_{ij})$ with $C_{ij}$ a block of dimension
$l_{i}\times l_{k}$ means 
\[
J_{l_{i}}C_{ij}=C_{ij}J_{l_{j}}.
\]
The dimension spanned by such blocks is $\min(l_{i},l_{j})$. Let
us call this dimension $d_{ij}$ and note that
\[
d_{ij}\le\frac{l_{i}l_{j}}{\max(l_{i},l_{j})}.
\]
 The matrix $A^{m}$ kills every Jordan cell of size $\le m$. Thus,
the dimension of the vector space spanned by the matrices of the form
$A^{m}B$ less than 
\[
\sum_{\max(l_{i},l_{j})\ge m}d_{ij}\le\sum\frac{l_{i}l_{j}}{m}=\frac{n^{2}}{m}.
\]

\end{proof}

\begin{proof}[Proof of Lemma \ref{lem: commutative matrices} ]
 By a standard agrument, we can assume that the matrices $A_{1},\,...,\, A_{l}$
are nilpotent. Namely, let us view the matrices $A_{1},\,...,\, A_{l}$
as operators on $V=\CC^{n}$. Since the algebra $\cF$ is commutative,
we can decompose the space $V$ into the direct sum of $\cF$-invariant
subspaces $V_{j}$ such that for every $A\in\cF$ and every $j$,
the eigenvalues of $A|_{V_{j}}$ coincide. We can restrict ourselves
to the case $V=V_{j}$, and substracting suitable constants from the
operators $A_{i}$, we may assume that all the $A_{i}$ are nilpotent.\\
 Let $x<n^{2}$, we will choose it later. Divide matrices of the form
$A_{1}^{j_{1}}A_{2}^{j_{2}}...A_{l}^{j_{l}}$ into two families. One
with $j_{i}<x$ for every $1\le i\le l$ and the other with at least
one of the powers $j_{i}\ge x$. The first family consists of $x^{l}$
matrices. Let us estimate the dimension of the subspace generated
by the second family. Suppose $j_{1}\ge x.$ The number of linearly
independent matrices of the form $A_{1}^{x}B$ where $A_{1}B=BA_{1}$
is at most $\frac{n^{2}}{x}$ by Lemma \ref{lem: jordan}. Thus $\dim\cF$
is bounded by 
\[
f(x):=\frac{ln^{2}}{x}+x^{l}.
\]
The minimum is achieved for $f'(x_{0})=0,\, x_{0}=n^{2/(l+1)}.$ We
obtain 
\[
\dim\cF\le ln^{2-2/(l+1)}+n^{2l/(l+1)}=(l+1)n^{2-\frac{2}{l+1}}.
\]

\end{proof}
I would like to thank Joseph Bernstein, Omer Ben-Neria and Dror Speiser
on useful discussions.\\
The research was supported by ISF Grant No. 1794/14.

\end{document}